\documentclass[11pt,a4paper,reqno]{amsart}
\usepackage{color, latexsym, amsfonts, amssymb, bbm, comment, amsmath, cite, amsthm, bbold}
\usepackage[foot]{amsaddr}

\usepackage{float}
\usepackage{indentfirst} % identation on the first paragraph
\usepackage{amsmath, amsthm, amssymb,xfrac, mathrsfs} % math packages
\usepackage{tikz} % figures
\usetikzlibrary{hobby}
\usetikzlibrary{decorations.pathreplacing}
\usepackage{wrapfig} % allows to put text aroung figures
\usepackage{xcolor} % colors
\usepackage{verbatim} % multiline comments
\usepackage{enumerate} % making lists

\usepackage[section]{algorithm} % algorithm packages
\usepackage{algpseudocode}

\usepackage{hyperref} % insert this allways last
\hypersetup{
colorlinks=true,
citecolor=black!50!red,
linkcolor=black!50!red,
linktoc=all
}
\usepackage[all]{hypcap} % to correct positioning for the hyperref

\newcounter{constant}
\newcommand{\newconstant}[1]{\refstepcounter{constant}\label{#1}}
\newcommand{\useconstant}[1]{c_{\textnormal{\tiny \ref{#1}}}}
\newcounter{bigconstant}

\newtheorem{teo}{Theorem}[section]
\newtheorem{prop}[teo]{Proposition}
\newtheorem{lemma}[teo]{Lemma}

\numberwithin{equation}{section} % numbering according to sections

\theoremstyle{definition}

\newtheorem{remark}[teo]{Remark}

\renewcommand{\P}{\mathbb{P}}
\newcommand{\E}{\mathbb{E}}
\newcommand{\R}{\mathbb{R}}

\newcommand{\N}{\mathbb{N}}
\newcommand{\Z}{\mathbb{Z}}

\DeclareMathOperator{\poisson}{Poisson}
\DeclareMathOperator{\ber}{Bernoulli}
\DeclareMathOperator{\expo}{Exponential}

\DeclareMathOperator{\var}{Var}
\DeclareMathOperator{\infl}{Inf}

\DeclareMathOperator{\cir}{Circ}
\DeclareMathOperator{\arm}{Arm}
\renewcommand{\d}{{\rm d}}
\newcommand{\pnorm}[2]{\left\|#2\right\|_{#1}}

\title[2d majority dynamics sharp threshold]{Sharp threshold for two-dimensional \\ majority dynamics percolation}

\author{Caio Alves}
\address{Institute of Mathematics, University of Leipzig -- Augustusplatz 10, 04109 Leipzig, Germany}
\email{caiotmalves@gmail.com}

\author{Rangel Baldasso}
\address{Bar-Ilan University, 5290002, Ramat Gan, Israel}
\email{r.baldasso@math.leidenuniv.nl}

\date{\today}

\begin{document}

	\begin{abstract}
	
	In this work we consider the two-dimensional percolation model arising from the majority dynamics process at a given time~$t\in\R_+$. We show the emergence of a sharp threshold phenomenon for the box crossing event at the critical probability parameter~$p_c(t)$ with polynomial size window. We then use this result in order to obtain stretched-exponential bounds on the one-arm event probability in the subcritical phase. Our results are based on differential inequalities derived from the OSSS inequality, inspired by the recent developments by Ahlberg, Broman, Griffiths, and Morris and by Duminil-Copin, Raoufi, and Tassion. We also provide analogous results for percolation in the voter model.
		
	\end{abstract}
	
	\maketitle
% Text

\section{Introduction}
~
\par In recent years, the study of sharp threshold phenomena in percolation has received great attention. This is mainly due to the development of new techniques that allow the treatment of dependent models~\cite{dcrt0,dcrt1,dcrt2}. Following this line, in this paper we prove that, for each fixed $t \geq 0$, percolation in two-dimensional majority dynamics undergoes a sharp phase transition in the  density parameter.

\par In two-dimensional majority dynamics, each vertex $x \in \Z^{2}$ receives an initial independent opinion which can be either zero or one\footnote{Following the usual notation in percolation theory, we refer to sites with opinion zero as closed and to sites with opinion one as open.}. With rate one, the vertex $x$ updates its opinion to match the majority of its neighbors. In the case of a tie, the original opinion is kept. Denote by $\P_{p,t}$ the distribution of the process at time $t$ when the initial density of ones is $p \in [0,1]$. Our interest lies in understanding the critical percolation function defined as
\begin{equation}\label{eq:perc_function}
p_{c}(t) = \inf \left\{ p \in [0,1]: \P_{p,t}\left[\begin{array}{c} \text{there exists an} \\ \text{infinite open path} \end{array} \right]>0\right\}.
\end{equation}

\par Not much is known about the behavior of the function above. In a work by Amir and the second author~\cite{ab}, it is proved that, for each $t>0$, $p_{c}(t) \in \left[\frac{1}{2}, p_{c}^{site}\right)$, where $p_{c}^{site}$ is the critical threshold for two-dimensional site percolation. Besides, the same work proves that $t \mapsto p_{c}(t)$ is a continuous non-increasing function and that there is no percolation at criticality for each $t \geq 0$.

\par Our main result here regards crossing events. For each $n \in \N$ and $\lambda >0$, let $R_{n}^{\lambda} = [1, \lambda n]\times [1,n]$, and consider the crossing event
\begin{equation}
H(\lambda n,n) = \left\{\begin{array}{c} \text{there exists an open path contained in } R_{n}^{\lambda} \\ \text{ that connects } \{1\} \times [1,\lambda n] \text{ to } \{n\} \times [1,\lambda n] \end{array} \right\}.
\end{equation}

\begin{teo}\label{t:sharp_thresholds}
For each $t \geq 0$, there exists $\gamma=\gamma(t)>0$ such that, for all $\lambda>0$,
\begin{equation}
\P_{p_{c}(t)-n^{-\gamma},t}[H(\lambda n,n)] \to 0 \quad \text{and} \quad \P_{p_{c}(t)+n^{-\gamma},t}[H(\lambda n,n)] \to 1,
\end{equation}
as $n$ grows.
\end{teo}

\begin{remark}\label{remark:lambda=1}
Even though we state the theorem above for general aspect ratio, we write the proof for the case $\lambda=1$ and denote $R_{n}^{1}$ simply by $R_{n}$. The proof remains the same for the case when $\lambda \neq 1$.
\end{remark}

%%%%%
\newconstant{c:decay}
%%%%%

\par As a consequence of Theorem~\ref{t:sharp_thresholds}, together with a general multiscale renormalisation argument, we obtain stretched-exponential decay of one-arm probabilities in the subcritical phase, together with analogous results for the supercritical case and dual closed $*$-paths\footnote{A $*$-path in $\Z^{2}$ is a path $x_{1}, x_{2}, \dots, x_{n}$ of vertices in $\Z^{2}$ such that $\pnorm{\infty}{x_{i+1}-x_{i}} = 1$, for all $i=1,2, \dots, n-1$. In other words, it is a path that is allowed to cross diagonals on the lattice $\Z^{2}$.}:

\begin{teo}\label{t:exp_decay}
For any $t \geq 0$, $\varepsilon>0$, and $p < p_{c}(t)$, there exists a positive constant $\useconstant{c:decay}=\useconstant{c:decay}(p, t, \varepsilon)>0$ such that
\begin{equation}
\label{eq:exp_decay1}
\P_{p,t}\left[\begin{array}{c} \text{there exists an open} \\ \text{ path connecting } 0 \text{ to the} \\ \text{  boundary of the ball } B(0,n) \end{array} \right] \leq \useconstant{c:decay}^{-1} \exp\left\{-\useconstant{c:decay}\frac{n}{(\log n)^\varepsilon}\right\}.
\end{equation}
Furthermore, if~$p > p_{c}(t)$, regarding long closed $*$-paths, we have: 
\begin{equation}
\label{eq:exp_decay2}
\P_{p,t}\left[\begin{array}{c} \text{there exists a closed} \\ \text{ $*$-path connecting } 0 \text{ to the} \\ \text{ boundary of the ball } B(0,n) \end{array} \right] \leq \useconstant{c:decay}^{-1} \exp\left\{-\useconstant{c:decay}\frac{n}{(\log n)^\varepsilon}\right\}.
\end{equation}
Also for~$p > p_{c}(t)$, we have
\begin{equation}
\label{eq:exp_decay3}
\P_{p,t}\left[\begin{array}{c} \text{The open cluster containing }0 \\ \text{is finite with diameter }n   \end{array} \right] \leq \useconstant{c:decay}^{-1} \exp\left\{-\useconstant{c:decay}\frac{n}{(\log n)^\varepsilon}\right\},
\end{equation}
and an analogous result holds for the closed $*$-connected cluster containing~$0$ when~$p < p_{c}(t)$.
\end{teo}

\begin{remark}
We strongly believe the poly-logarithmic correction present in the exponent of the inequalities above to be a shortcoming of the renormalisation argument used, and conjecture the correct bound to be simply exponential in~$n$.
\end{remark}

\begin{remark}
The above result follows from a general statement inspired by~\cite{pt} that we prove here about dependent percolation with fast decay of correlations. Under general conditions (see Proposition~\ref{p:1armdecaygen}), which are provided by Theorem~\ref{t:sharp_thresholds} and decoupling inequalities, this statement implies stretched-exponential decay of the one-arm event's probability whenever the probability of crossing a large annulus is sufficiently small.
\end{remark}

\begin{remark}
Regarding crossing functions and one-arm events for $p=p_{c}(t)$, in~\cite{ab} it is proved that the probability of $H(\lambda n,n)$ is bounded away from zero and one uniformly in $n$. This follows by combining Theorem 1.1 and Lemma 4.6 from that paper, together with a RSW theory for $*$-crossings analogous to the one that can be found in~\cite{att}. Furthermore, the decay of the one-arm event is polynomial in $n$.
\end{remark}

\bigskip

\noindent \textbf{Overview of the proofs.} The proof of Theorem~\ref{t:sharp_thresholds} relies on exploiting the relation between Boolean functions and randomized algorithms obtained through OSSS inequality. Here it is possible to write the existence of a crossing at time $t$ as a random Boolean function of the initial configuration, with randomness coming from the evolution of majority dynamics. A first approach would then be to consider the quenched configuration, where the clocks of the Markov process are fixed, and try to use these tools directly on the space of initial configurations, for each possible realization of the Poisson clocks in the evolution. This idea fails, since quenched configurations lack the homogeneity needed for our arguments.

\par To circumvent this difficulty, we need to consider the randomness that comes from the evolution together with the one from the initial configuration. We then revisit the idea developed in~\cite{abgm}, and further explored in~\cite{att} and~\cite{ahlbergb}, of using a two-stage construction of the process to obtain a discretization of it that still retains relevant properties of the annealed evolution. The central idea is to construct the process in a way that each vertex is associated to a Poisson point process of clocks of intensity $k \in \N$, with~$k$ large. Whenever a clock in a given vertex ticks, we keep this tick with probability $\frac{1}{k}$ and, in this case, update the opinion of the vertex to agree with the majority of its neighbors.

\par This artificial increase of the density of clock ticks allows us now to consider quenched probabilities, as we condition on the denser Poisson process, and still retain good properties of the annealed configuration with large probability. Given a collection of clock ticks, we obtain a Boolean function by considering the initial opinions and the selection of the clock ticks that are kept for the evolution.

\par We then proceed to analyze this quenched random Boolean function. First, we devise an algorithm that determines the outcome of the function and bound its revealment. This algorithm is a simple exploration process that discovers the open components that intersect a random line crossing the rectangle $R_{n}$ by querying the initial state of sites and which clock ticks are selected to compose the evolution. The bound on the revealment will follow from one-arm estimates in the quenched setting (see Proposition~\ref{prop:one_arm}). These estimates in turn are derived from Russo-Seymour-Welsh-type results stated in~\cite{ab} and inspired by~\cite{tassion}.

\par Since we are considering randomness that comes from the time evolution as well, when applying OSSS inequality it will be necessary to control the influence of clock ticks. We relate time-pivolality to space-pivotality, bounding the influence of a clock tick by a combination of the influences of the initial positions (see Proposition~\ref{prop:piv}). This pivotality relation is the most original and sensitive part of our proof, and fails, for example, if one considers the contact process instead of majority dynamics as the rule for the time evolution of the opinions. Nevertheless, we can also prove a similar result for the voter model (see Section~\ref{sec:further_models}). With this relation in hands, we are able to conclude the proof of Theorem~\ref{t:sharp_thresholds}.

\par Let us now turn our attention to the proof of Theorem~\ref{t:exp_decay}. Here, we provide a general statement on the decay rate of the one-arm probability in percolation models with fast decaying correlations. We prove that, provided the annulus crossing probability goes to~$0$ as the size of the annulus goes to infinity, the rate of decay of the one-arm probability is at least stretched exponential in the ball's radius. Combining this with Theorem~\ref{t:sharp_thresholds} yields Theorem~\ref{t:exp_decay}. The proof of this statement relies on a multiscale renormalisation argument adapted from~\cite{pt}. 

\begin{remark}
Our technique is somewhat general and might be applied to other dynamics. As an example, in Section~\ref{sec:further_models}, we explain how to adapt it to the case when the opinions follow the voter model. The greatest obstacle to a broader generalization is the lemma relating time- and space-pivotality, whose proof is strongly model-dependent.
\end{remark}

\begin{remark} Camia, Newmann, and Sidoravicius in~\cite{cns} prove that fixation of the opinions happens with stretched-exponential speed in a sub-interval of the supercritical phase. The idea of the proof is to observe that, if $p$ is larger than $p_{c}^{site}$ (the critical probability for Bernoulli site percolation in~$\Z^2$), one can obtain a random partition of $\Z^{2}$ into finite subsets whose boundaries are circuits of constant initial opinion which are preserved by the dynamics, reducing the evolution to finite random subsets. This, together with the uniform bound on the number of changes in opinion each vertex can have (see Tamuz and Tessler~\cite{tt}), allows one to conclude that the speed of convergence is stretched exponential. They further improve the proof by performing an enhancement on the initial configuration, and conclude that stretched exponential decay also holds for values of $p$ slightly smaller than $p_{c}^{site}$. We remark that the same idea can be applied together with Theorem~\ref{t:exp_decay} to verify that stretched-exponential decay of the non-fixation probability also holds for $p \in \left( \lim p_{c}(t), 1 \right]$. Symmetry considerations imply an analogous result for $p \in  \left[0, 1- \lim p_{c}(t) \right)$.
\end{remark}

\bigskip

\noindent \textbf{Related works.} Russo's approximate 0-1 law~\cite{russo} is one of the first results regarding sharp thresholds in independent percolation. It says that a sequence of monotone Boolean functions exhibits a sharp threshold, provided the supremum of the influences converges to zero. The use of randomized algorithms and OSSS inequality to understand threshold phenomena is much more recent and so far has proven to be a very powerful technique. Duminil-Copin, Raoufi and Tassion~\cite{dcrt1, dcrt2} use these techniques to study the subcritical phase of Voronoi percolation and threshold phenomena for the random-cluster and Potts models, while~\cite{dcrt0}, by the same authors, considers the case of Boolean percolation, under moment conditions of the radii distribution.

\par After these seminal works, other applications of such techniques were found. Muirhead and Vanneuville~\cite{mv} use this approach to conclude that level-set percolation for a wide class of smooth Gaussian processes undergoes a sharp phase transition. Dereudre and Houdebert~\cite{dh} conclude similar statements for the Widom-Rowlinson model.

\par The collection of upper invariant measures for the contact process was also studied. Van den Berg~\cite{vdb} considers the two-dimensional case, and proves the existence of a sharp phase transition without relying on the OSSS inequality, but on Talagrand inequality instead.

\par The discretization we use here is more in line with the one considered in Ahlberg, Broman, Griffiths, and Morris~\cite{abgm}, where the authors prove noise sensitivity for the critical Boolean model. With a similar discretization, and relying on Talagrand's inequality~\cite{talagrand}, Ahlberg, Tassion, and Teixeira~\cite{att}~deduce that Boolean percolation undergoes a sharp phase transition. Furthermore, Ahlberg, in collaboration with the second author~\cite{ahlbergb}, employs this technique to study noise sensitivity of two-dimensional Voronoi percolation and conclude, as a corollary, the existence of a sharp threshold with polynomial window.

\bigskip

\noindent \textbf{Open problems.}  Regarding the percolation function $p_{c}(t)$ (see Equation~\eqref{eq:perc_function}), it is known that it is a continuous non-increasing function that is strictly decreasing at zero. Whether or not it is strictly decreasing in the whole non-negative real line it is still not known. We hope our new estimates on the connectivity decay of the subcritical phase might help. Regarding its asymptotic behavior, we conjecture that $p_{c}(t)$ converges to $\frac{1}{2}$, as $t$ grows. From~\cite{tt}, one obtains that, almost surely, the process has a limiting configuration $\eta_{\infty}$. General results on two dimensional percolation imply that $\eta_{\infty}$ does not percolate for $p = \frac{1}{2}$ (see Gandolfi, Keane, and Russo~\cite{gkr}).

\par Our techniques are reliant on RSW theory, and are therefore limited to two dimensions. We believe our results to be valid for any dimension and for a large class of particle system models, and that with future developments in the field such general problems will be tractable. An interesting process where this should give some insight is zero-temperature Glauber dynamics for the Ising model. Here, the main difficulty is in relating time- and space-pivotality. We intend to pursue this in a future work.

\par Another problem this work leaves open is the correct decay of the one-arm probabilities in Theorem~\ref{t:exp_decay}, which we conjecture to be simply exponential in the distance~$n$.

\bigskip

\noindent \textbf{Organization of the paper.} In Section~\ref{sec:properties}, we state properties of majority dynamics and some results that will be used throughout the text. Section~\ref{sec:construction} contains a graphical construction of majority dynamics that will be used in our results, while Section~\ref{sec:influence} discusses the concept of influences and pivotality in the quenched setting. We present a randomized algorithm and bound its revealment in Section~\ref{sec:alg}, and use this algorithm to conclude the proof of Theorem~\ref{t:sharp_thresholds} in Section~\ref{sec:thresholds}. In Section~\ref{sec:one_arm}, we provide quenched one-arm estimates for the model that were previously assumed in the proof of Theorem~\ref{t:sharp_thresholds}. Theorem~\ref{t:exp_decay} is proved in Section~\ref{sec:decay}. Finally, we discuss how to modify our result to the case when the dynamics follows the voter model in Section~\ref{sec:further_models}.

\bigskip

\noindent \textbf{Acknowledgments.} The authors thank Daniel Ahlberg, Augusto Teixeira, and Daniel Valesin for valuable discussions and improvements during the elaboration of this work. CA is supported by the DFG grant SA 3465/1-1. RB is supported by the Israel Science Foundation through grant 575/16 and by the German Israeli Foundation through grant I-1363-304.6/2016.

\section{Basic properties}\label{sec:properties}
~
\par We denote by~$\eta\equiv\eta(p)=(\eta_t)_{t \in \R{+}}$ the two-dimensional majority dynamics process with initial configuration~$\eta_0\in\{0,1\}^{\Z^2}$, which assigns i.i.d.\ $\mathrm{Bernoulli}(p)$ random variables to each vertex of~$\Z^2$. As mentioned in the Introduction, we denote by~$\P_{p,t}$ the law of~$\eta_t=\eta_t(p)$. We collect here facts about this collection of measures. Complete proofs can be found in~\cite{ab} and references therein.

\par Notice that, as a consequence of Harris~\cite{harris} and a correlation decay estimate (see Equation~\ref{eq:correlation_decay_radius}) used to extend the result in~\cite{harris} to countable state space, the measures $\P_{p,t}$ are positively associated. This is the same as stating that $\P_{p,t}$ satisfies the FKG inequality: for any two events $A$ and $B$ that are increasing with respect to the partial ordering\footnote{We say $\eta \preceq \xi$ if $\eta(x) \leq \xi(x)$, for all $x \in \Z^{2}$. An event $A$ is increasing with respect to this partial ordering if $\eta \in A$ and $\eta \preceq \xi$ imply $\xi \in A$.} of $\{0,1\}^{\Z^{2}}$, it holds that
\begin{equation}
\P_{p,t}[A \cap B] \geq \P_{p,t}[A] \P_{p,t}[B].
\end{equation}

\par Given two disjoint subsets $A$ and $B$ of $\Z^{2}$ and $X \subset \Z^{2}$ such that $A \cup B \subset X$, we define the event
\begin{equation}\label{eq:open_connection}
\left\{A  \overset{X}{\longleftrightarrow} B\right\}
\end{equation}
as the existence of an open path contained in $X$ connecting a vertex in $A$ to a vertex in $B$. We omit $X$ in the notation above when $X=\Z^{2}$. The event where percolation holds is defined as the existence of an infinite open path. Standard arguments yield that
\begin{equation}
\P_{p,t}[\eta \text{ percolates}]>0 \quad \text{if, and only if,} \quad \inf_{n}\P_{p,t}\left[\{0\} \leftrightarrow \partial B(0,n)\right]>0,
\end{equation}
where $\partial B(0,n) = \{x \in \Z^{2}: \pnorm{\infty}{x}=n\} $ is the boundary of the ball $B(0,n)=[-n,n]^{2}$.

%%%%%
\newconstant{c:correlation}
\newconstant{c:rsw}
%%%%%

\par Let us now list some properties of the probabilities $\P_{p,t}$ for a fixed $t$. First of all, we state correlation decay for these measures, which is a consequence of standard cone-of-light estimates (see Propositions~\ref{prop:col} and~\ref{prop:decoup} below). For each $t \geq 0$, there exists a constant $\useconstant{c:correlation}=\useconstant{c:correlation}(t)$ such that, if $A$ is an event that depends on the configuration $\eta_{t}(x)$ only on sites inside $[-n,n]^{2}$ and $B$ is an event that depends on the configuration on sites outside $[-2n,2n]^{2}$, then, for every $p \in [0,1]$,
\begin{equation}\label{eq:correlation_decay_radius}
\Big|\P_{p,t}[A \cap B] - \P_{p,t}[A]\P_{p,t}[B] \Big| \leq \useconstant{c:correlation}n^{2}e^{-\frac{n}{8}\log n}.
\end{equation}

\par Given $\lambda>0$, denote by $H(\lambda n, n)$ the crossing event
\begin{equation}
H(\lambda n,n) = \left[ \{1\} \times [1,n] \overset{R_{n}}{\longleftrightarrow} \{\lfloor \lambda n \rfloor \} \times [1,n] \right],
\end{equation}
where $R_{n}\equiv R_n(\lambda)=[1,\lambda n] \times [1,n]$, and let $H^{*}(\lambda n, n)$ denote the event of the existence of a closed horizontal $*$-crossing of the rectangle $R_{n}$. The main result regarding crossing events is the RSW theory, that we can obtain by adapting the proofs of Tassion~\cite{tassion}, since they rely on the invariance of the percolation measure under certain symmetries of~$\Z^2$, decay of correlations, bounds for crossings of squares, and the FKG inequality, properties that are also available to us.

\begin{prop}[RSW theory]\label{prop:RSW}
For each fixed value of $t \geq 0$ and each $\lambda>0$, there exists a positive constant $\useconstant{c:rsw}=\useconstant{c:rsw}(\lambda, t)>0$ such that
\begin{equation}\label{eq:rsw}
\useconstant{c:rsw} \leq \P_{p_{c}(t),t}\left[H(\lambda n,n) \right] \leq 1-\useconstant{c:rsw},
\end{equation}
for all $n \in \N$.
\end{prop}

\par Since $H(\lambda n, n)$ holds if, and only if, there is no closed vertical $*$-crossing of $R_{n}=[1,\lambda n] \times [1,n]$, one can easily deduce from the proposition above that an analogous result holds for the event $H^{*}(\lambda n, n)$. Furthermore, monotonicity considerations imply that, for all $p \geq p_{c}(t)$,
\begin{equation}
\label{eq:rsw2}
\inf_{n} \P_{p,t}\left[H(\lambda n,n) \right] \geq \useconstant{c:rsw}(\lambda, t),
\end{equation}
and, for all $p \leq p_{c}(t)$,
\begin{equation}
\label{eq:rsw3}
\inf_{n} \P_{p,t}\left[H^{*}(\lambda n,n) \right] \geq \useconstant{c:rsw}(\lambda^{-1}, t).
\end{equation}

Since this is a straighforward adaptation of the proof in~\cite{tassion}, we choose to omit it here.

\bigskip

\noindent \textbf{The OSSS inequality.} Let us quickly recall the version of the OSSS inequality we use here. Fix $f:\{0,1\}^{n} \to \{0,1\}$ a Boolean function and, for a vector $\bold{p} = (p_{1}, \dots p_{n})$, let $\P_{\bold{p}}$ denote the probability measure on $\{0,1\}^{n}$ where each entry is independent and the $i$-th entry has probability $p_{i}$ of being one. For each $i \in [n]$, we define the influence of the bit $i$ as
\begin{equation}
\infl_{\bold{p}}(f,i) = \P_{\bold{p}}[f(\omega) \neq f(\omega^{i})],
\end{equation}
where $\omega^{i}$ is obtained from $\omega$ by changing the $i$-th entry of the vector.

\par A (randomized) algorithm $\mathcal{A}$ is a rule that outputs a value zero or one, by querying entries of the vector $\omega$, and whose choice of the next entry to be queried is allowed to depend on the previous observations. An algorithm can determine its output before querying all bits, in this case we say the algorithm~\emph{stops}. We say that the algorithm determines $f$ if its outcome coincides with $f(\omega)$, for every $\omega$. The revealment of the bit $i$ for an algorithm $\mathcal{A}$ is the quantity
\begin{equation}
\delta(\mathcal{A},i) = \P_{\bold{p}}[\mathcal{A} \text{ queries } i \text{ before stopping}].
\end{equation}

\par The OSSS inequality (see~\cite{osss}, Theorem 3.1) provides the bound
\begin{equation}
\label{eq:osss}
\var(f) \leq \sum_{i=1}^{n}\delta(\mathcal{A},i)\infl_{\bold{p}}(f,i).
\end{equation}
Poincar\'{e}'s inequality can be recovered from the above inequality by bounding all the revealments by one.

\section{The two-stage construction}\label{sec:construction}
~
\par In this section we present a graphical construction of majority dynamics that will be used in the rest of the paper. We begin by presenting the usual Harris construction, since we will use a simple modification of it.

\par Consider a collection $\mathscr{P}=\big( \mathscr{P}_{x} \big)_{x \in \Z^2}$ of i.i.d. Poisson processes in the interval~$[0,t]$ with rate one. For each $x \in \Z^2$, the clocks $\mathscr{P}_{x}$ will control the updates in that site: whenever the clock at $x$ ticks, the opinion at $x$ is updated to match the majority of its neighbors. In case of a draw, the site keeps its original opinion. With this construction, we can fully determine the state of the system at any given time with the collection of clocks $\mathscr{P}$ and the initial configuration $\eta_{0}$. This is a classical fact that follows from the observation that the only vertices whose initial opinions are necessary in order to determine $\eta_{s}(x)$, for each $x \in \Z^{2}$ and $s \leq t$, are the ones which are connected to $x$ via a path of vertices with clocks that ring in increasing order. This set of points is easily seen to be almost surely finite. This same fact is at the heart of the cone-of-light estimates presented below (see Proposition~\ref{prop:col}).

\begin{remark}\label{remark:voter_model}
It is possible to obtain the voter model with the same graphical construction, just by modifying the way sites are updated: instead of choosing the new opinion to be the majority of the neighboring opinions, the update is made by copying the opinion of a randomly selected neighbor.
\end{remark}

\par The construction we will use is a slight modification of the one presented above. Instead of considering the collection of clocks $\mathscr{P}$, we start with a denser collection of clocks $\mathscr{P}^{k}=\big( \mathscr{P}_{x}^{k} \big)_{x \in \Z^2}$ distributed as i.i.d. Poisson processes on the interval~$[0,t]$ with rate $k$, where $k$ is a fixed positive integer number that will be taken to be large. With this collection of clocks in hand, we need some additional randomness in order to define the process: whenever a clock ticks, we perform the update at the respective site with probability $\frac{1}{k}$ (this can be realized by considering an independent $\ber\left(\frac{1}{k}\right)$ random variable for each clock tick of $\mathscr{P}^{k}$). In this case, conditioned on the realization of the clocks $\mathscr{P}^{k}$, we can obtain the state of the system at any given time $t \geq 0$ by using the initial configuration $\eta_{0}$ and the collection of random variables that verify whether or not each update is performed. We will denote by $\P_{kt}$ the distribution of $\mathscr{P}^{k}$ and by $\P_{p, \frac{1}{k}}$ the joint distribution of the initial condition and the additional randomness necessary in order to determine the process~$(\eta_{s})_{s \geq 0}$.

\par The advantage of the last construction presented above lies in the fact that the model at time $t$ may be seen as a random Boolean function: for each realization of $\mathscr{P}^{k}$, we obtain a Boolean function whose entries select the initial configuration and which updates are performed. By choosing the value of $k$ large enough, we can ensure that these random functions are well-behaved, in a sense that we will make clear later.

\par We will work with the process conditioned on the realization $\mathscr{P}^{k}$. In this case, we may write the characteristic function of the crossing event $\left[ \eta_{t} \in H(\lambda n,n)\right]$ as a Boolean function $f_{n}: \{0,1\}^{\Lambda} \to \{0,1\}$, where 
\[
\Lambda = \Z^2 \cup \{(x,s): x \in \Z^2, \, s \in \mathscr{P}^{k}_{x} \cap [0,t]\},
\]
and such that each configuration describes the entries at time zero and whether each clock tick before time $t$ is accepted or not. We will denote a configuration on $\{0,1\}^{\Lambda}$ by a pair $(\eta_{0}, \mathscr{P})$, where the first coordinate contains the initial opinions of each site and the second retains the information of which clock ticks are kept. Moreover, each entry of $\eta_{0}$ will be distributed as a $\ber(p)$ random variable, where $p \in [0,1]$ is the initial density of the process, and each entry of $\mathscr{P}$ will have distribution $\ber\left(\frac{1}{k}\right)$.

\par Since, almost surely (on $\mathscr{P}^{k}$), one needs to observe only a finite amount of sites in order to verify if $H(\lambda n,n)$ holds or not, the domain of $f_{n}$ is almost surely finite and hence this is a well-defined Boolean function.

\par The main reason we consider this construction is the following lemma.
\begin{lemma}\label{lemma:variance_decay}
For every integer $k \geq 2$, $p \in (0,1)$ and Boolean function $f$ of the graphical construction, we have
\begin{equation*}%\label{eq:varianve_decay}
\var\Big(\E\left[f(\eta_{0}, \mathscr{P})\middle|\mathscr{P}^{k}\right]\Big) \leq \frac{1}{k}.
\end{equation*}
\end{lemma}

\begin{proof}
The proof follows simply by considering a particular construction of the pair $(\mathscr{P},\mathscr{P}^{k})$. First, let $\mathscr{P}_{1},\mathscr{P}_{2},\ldots,\mathscr{P}_{k}$ be independent copies of $\mathscr{P}^{1}$ (and independent of $\eta_{0}$), and consider $\kappa$ be chosen uniformly in $[k]\equiv\{1,\dots,k\}$. Observe that $(\mathscr{P},\mathscr{P}^{k}) \sim (\mathscr{P}_{\kappa}, \cup_{i \in [k]} \mathscr{P}_{i})$. From this, one readily obtains the equality
\begin{equation*}
\var\left(\E\left[f(\eta_{0}, \mathscr{P})\middle|\mathscr{P}^{k}\right]\right) = \var\Big(\E\Big[f(\eta_{0}, \mathscr{P}_{\kappa})\Big|\cup_{i \in [k]}\mathscr{P}_{i}\Big]\Big)
\end{equation*}

Directly from Jensen's inequality we obtain\footnote{If $\mathcal{F} \subset \mathcal{G}$ are two $\sigma$-algebras, Jensen's inequality implies
\begin{equation*}
\E \Big[X\Big|\mathcal{F}\Big]^{2} = \E \Big[\E[X|\mathcal{G}]\Big|\mathcal{F}\Big]^{2} \leq \E\Big[\E[X|\mathcal{G}]^{2}\Big|\mathcal{F}\Big],
\end{equation*}
from where one deduces that
\begin{equation*}
\var\left(E[X|\mathcal{F}]\right) = \E\left[E[X|\mathcal{F}]^{2}\right] - E[X]^{2} \leq \E\left[E[X|\mathcal{G}]^{2}\right] - E[X]^{2} = \var\left(E[X|\mathcal{G}]\right).
\end{equation*}
Estimate~\eqref{eq:variance_inequality} follows directly.}
\begin{equation}\label{eq:variance_inequality}
\var\Big(\E\Big[f(\eta_{0}, \mathscr{P}_{\kappa})\Big|\cup_{j \in [k]}\mathscr{P}_{j}\Big]\Big) \leq \var \Big(\E\left[f(\eta_{0}, \mathscr{P}_{\kappa})\middle|(\mathscr{P}_{i})_{i=1}^{k}\right]\Big)
\end{equation}

The conditional expectation on the variance above can be easily calculated
\begin{equation*}
\E\left[f(\eta_{0}, \mathscr{P}_{\kappa})\middle|(\mathscr{P}_{i})_{i=1}^{k}\right] = \frac{1}{k}\sum_{i=1}^{k}\E_{\eta_0}[f(\eta_{0}, \mathscr{P}_{i})],
\end{equation*}
where $\E_{\eta_{0}}$ denotes the expectation with respect to the initial condition $\eta_{0}$.

Finally, since the processes $(\mathscr{P}_{i})_{i=1}^{k}$ are independent, we obtain
\begin{equation*}
\var\bigg(\frac{1}{k}\sum_{i=1}^{k}\E_{\eta_0}[f_{n}(\eta_{0}, \mathscr{P}_{i})]\bigg) \leq \frac{1}{k},
\end{equation*}
concluding the proof.
\end{proof}

%%%%%
\newconstant{c:cir}
%%%%%

\par We can use the above lemma together with RSW theory to bound quenched probabilities in good events. Let
\begin{equation}\label{eq:cir}
\cir(n) = \left\{ \begin{array}{c} \text{there exists an open circuit} \\ \text{contained in } B\left(0, 3n \right) \setminus B\left(0, n \right) \end{array}\right\},
\end{equation}
and write $\cir^{*}(n)$ for the equivalent event, but asking for the existence of a closed $*$-circuit. Notice that Equations~(\ref{eq:rsw2}) and~(\ref{eq:rsw3}) and the FKG inequality imply that there exists a positive constant $\useconstant{c:cir}=\useconstant{c:cir}(t)>0$ such that
\begin{equation}\label{eq:RSW_cir}
\inf_{n}\P_{p,t}\left[\cir(n)\right] \geq \useconstant{c:cir},
\end{equation}
if $p \geq p_{c}(t)$, and
\begin{equation}\label{eq:RSW_cir2}
\inf_{n}\P_{p,t}\left[\cir^{*}(n)\right] \geq \useconstant{c:cir},
\end{equation}
for $p \leq p_{c}(t)$.
\begin{lemma}\label{lemma:quenched_cir}
For any fixed $t \geq 0$ and $k \geq 2$,
\begin{equation*}
\P_{kt} \Big[\P_{p , \frac{1}{k}}\left[ \cir(n)\middle| \mathscr{P}^{k}\right]  \leq \frac{\useconstant{c:cir}}{2} \Big] \leq \frac{4}{\useconstant{c:cir}^{2}k},
\end{equation*}
for all $n \geq 1$ and $p \geq p_{c}(t)$. An analogous estimate holds for $\cir^{*}(n)$ if $p \leq p_{c}(t)$. 
\end{lemma}

\begin{proof}
The same proof of Lemma~\ref{lemma:variance_decay} can be used to the characteristic function of $\cir(n)$. Combining this with Chebyshev inequality and~\eqref{eq:RSW_cir} implies
\begin{equation*}
\begin{split}
\P_{kt} & \Big[\P_{p, \frac{1}{k}}\left[\cir(n)\middle| \mathscr{P}^{k}\right]  \leq \frac{\useconstant{c:cir}}{2} \Big]\\
& \qquad\leq 
\P_{kt} \left[\Big|\P_{p, \frac{1}{k}}\left[\cir(n)\middle| \mathscr{P}^{k}\right]-\P_{p, t}[\cir(n)]\Big| \geq \frac{\useconstant{c:cir}}{2}\right] \\
& \qquad \leq \frac{4}{\useconstant{c:cir}^{2} k},
\end{split}
\end{equation*}
for all $k \geq 2$. To conclude the statement for $\cir^{*}(n)$, one proceeds in the same way, but with~\eqref{eq:RSW_cir2} instead of~\eqref{eq:RSW_cir}.
\end{proof}

\par The result above provides quenched estimates for the existence of circuits and can be applied to deduce quenched one-arm estimates. For each $n$, let $\arm_{\sqrt n}(\eta_{0}, \mathscr{P})$ denote the event that there exists an open path connecting the boundary of the ball $B\left(0, n^{\sfrac{1}{4}}\right)$ to the boundary of the ball $B\left(0, n^{\sfrac{1}{2}}\right)$. This path can be chosen to be entirely contained inside $B\left(0, n^{\sfrac{1}{2}}\right) \setminus B\left(0, n^{\sfrac{1}{4}}\right)$. Denote also by $\arm^{*}_{\sqrt{n}}(\eta_{0}, \mathscr{P})$ the corresponding event, but asking for a closed $*$-path with the same properties.
\begin{prop}[One-arm estimate]\label{prop:one_arm}
There exists $\nu>0$ such that, for all $\gamma>0$, there exists $k_{0} \geq 2$ such that, for any $k \geq k_{0}$ and $p \leq p_{c}(t)$, if $n \geq n_{0}=n_{0}(k)$, then
\begin{equation}
\P_{kt}\Big[\P_{p, \frac{1}{k}}\left[\arm_{\sqrt n}(\eta_{0}, \mathscr{P})\middle|\mathscr{P}^{k}\right] \geq n^{-\nu} \Big] \leq n^{-\gamma}.
\end{equation}
An analogous result holds for $\arm^{*}_{\sqrt{n}}(\eta_{0}, \mathscr{P})$ instead of $\arm_{\sqrt n}(\eta_{0}, \mathscr{P})$ if we assume $p \geq p_{c}(t)$.
\end{prop}

\par The proof of the above Proposition relies on observing that $\arm_{\sqrt n}(\eta_{0}, \mathscr{P})$ holds if, and only if, there is no closed $*$-circuit inside $B\left(0, n^{\sfrac{1}{2}}\right) \setminus B\left(0, n^{\sfrac{1}{4}}\right)$. Since it is possible to find a logarithmic amount of disjoint and distant annuli in this set, we can repeatedly apply Lemma~\ref{lemma:quenched_cir} to obtain that the probability of not having such a circuit in any of the annuli is small. A complete proof requires additional care to control dependencies between the disjoint annuli, and we postpone it to Section~\ref{sec:one_arm}.

\par To conclude this section, we present cone-of-light estimates for the denser collection of clock ticks. Given~$\mathscr{P}^{k}$ we define the (past) cone of light $C_{k,t}^{\leftarrow}(x)$ to be the collection of vertices one needs to observe in order to determine $\eta_{s}(x)$, for all $s \in [0,t]$, varying over every possible pairs~$(\eta_{0}, \mathscr{P})$ of initial configurations and clocks selections. We also define the future cone of light $C_{k,t}^{\rightarrow}(x)$ as the set of vertices that can be influenced by~$x$ up to time~$t$, that is,
\begin{equation}
C_{k,t}^{\rightarrow}(x):=\{y\in\Z^2; x\in  C_{k,t}^{\leftarrow}(y)       \}.
\end{equation}
\begin{prop}[Cone-of-light estimates]\label{prop:col}
Given $k \in \N$ and $t \geq 0$, if $n$ is large enough,
\begin{equation}
\begin{split}
\P_{kt}\Big[C_{k,t}^{\leftarrow}(x) \cap \partial B(x,n) \neq \emptyset \Big] &\leq e^{-\frac{1}{8}n \log n},
\\
\P_{kt}\Big[ C_{k,t}^{\rightarrow}(x) \cap \partial B(x,n) \neq \emptyset \Big] &\leq e^{-\frac{1}{8}n \log n}.
\end{split}
\end{equation}
\end{prop}

\begin{proof}
Without loss of generality, we consider $x=0$ and prove the bound for~$C_{k,t}^{\leftarrow}(x)$, the other bound following by analogous reasoning. Notice that, in order for $C_{k,t}^{\leftarrow}(0)$ to intersect $\partial B(0, n)$, it is necessary that there exists a path of length at least~$n$ whose vertices' associated Poisson clocks ring in decreasing order. That is, there must exist a (not necessarily simple) path~$0=x_0,x_1,\dots,x_m\in\partial B(0,n)$, $m\geq n$, and a sequence of times
\begin{equation*}
t \geq t_0>t_1>\dots>t_m \text{ such that }t_j \text{ is a mark in }\mathscr{P}^{k}_{x_j}.
\end{equation*}
Combining the fact that these clocks are i.i.d with distribution $\expo(k)$, the relation between Poisson and Exponential distributions, and union bounds, we obtain
\begin{equation}
\begin{split}
\lefteqn{\P_{kt}\Big[C_{k,t}^{\leftarrow}(x) \cap \partial B(x,n) \neq \emptyset \Big] }\phantom{********}
\\
&\leq \sum_{m\geq n}\P_{kt} \left[  \begin{array}{c} \text{there exists a path of size } m \\ \text{starting at $0$ such that all clocks} \\ \text{ring before time $t$ in decreasing order} \end{array} \right] \\
& \leq \sum_{m\geq n} 4^m \P[\poisson(kt) \geq m] \\
& = e^{-kt}\sum_{m \geq n} 4^m \sum_{j \geq m} \frac{(kt)^{j}}{j!} 
\\
&\leq e^{-kt}\sum_{m \geq n}   e^{kt} \frac{(4tk)^{m}}{m!} \\
& \leq e^{4tk} \frac{(4tk)^{n}}{\left(\frac{n}{2}\right)^{\frac{n}{2}}}\\
& \leq e^{-\frac{1}{8} n \log n},
\end{split}
\end{equation}
if $n \geq \left(16 kt \right)^{8}$. This concludes the proof. 
\end{proof}

%%%%%
\newconstant{c:cor_decay}
%%%%%

\par The bound in~\eqref{eq:correlation_decay_radius} follows directly from Proposition~\ref{prop:col}. We will later need a stronger decoupling inequality in order to obtain bounds on the one-arm event's probability. The following proposition provides a generalized form of~\eqref{eq:correlation_decay_radius}:
\begin{prop}
	\label{prop:decoup}
	For every $t \geq 0$, there exists a positive constant $\useconstant{c:cor_decay}>0$ depending on~$t$ such that the following holds: for~$L,R>0$ and any pair of events $A$ and $B$ with respective supports inside the sets $x + [-L,2L]^2$ and $y + [-L,2L]^2$, with $\pnorm{\infty}{x-y} \geq 3L+R$, and $p \in [0,1]$,
	\begin{equation}\label{eq:strong_correlation_decay}
	\Big|\P_{p,t}[A \cap B] - \P_{p,t}[A]\P_{p,t}[B] \Big| \leq \useconstant{c:cor_decay}^{-1}L^{2}e^{-\useconstant{c:cor_decay}R \log R}.
	\end{equation}
\end{prop}

\begin{proof}
	If $C_{1,t}^{\leftarrow}(z) \cap \partial  B\left(z,\frac{R}{2}\right) = \emptyset$ for all $z \in  (x + [-L,2L]^2) \cup (y + [-L,2L]^2)$, then the occurrence of $A$ and $B$ are determined by disjoint (and hence independent) parts of the graphical construction.
	Defining the event
	\begin{equation*}
	\left\{\begin{array}{c} \text{For every } z \in  (x + [-L,2L]^2) \cup (y + [-L,2L]^2), \\ C_{1,t}^{\leftarrow}(y) \cap \partial B\left(y,\frac{R}{2}\right) = \emptyset \end{array}\right\} =: \mathrm{Dec}(L,R,x,y),
	\end{equation*}
	we obtain
	\begin{equation}
	\P_{p,t}\left[\mathrm{Dec}(L,R,x,y)^C\right] \leq 18 L^{2}e^{\frac{1}{16}R \log \frac{R}{2}},
	\end{equation}
	where the last bound above is a consequence of Proposition~\ref{prop:col} for large values of $R$. For sufficiently large~$R$, by intersecting the events~$A$, $B$, and $A\cap B$ with~$\mathrm{Dec}(L,R,x,y)$ and~$\mathrm{Dec}(L,R,x,y)^C$, we can show that
	\begin{equation*}
	\Big|\P_{p,t}[A \cap B] - \P_{p,t}[A]\P_{p,t}[B] \Big| \leq 5 \P_{p,t}\left[\mathrm{Dec}(L,R,x,y)^C\right] \leq 90 L^{2}e^{\frac{1}{16}R \log \frac{R}{2}}
	\end{equation*}
	Choosing the constant in~\eqref{eq:strong_correlation_decay} to cover the cases where~$R$ is not large enough concludes the proof.
\end{proof}

\section{Influence and space-pivotality}\label{sec:influence}
~
Given a realization of $\mathscr{P}^{k}$, the quenched influence of a bit $x \in \Z^2$ or $(x,s) \in \{x\} \times \mathscr{P}^{k}_{x}$ is defined respectively as
\begin{equation}
\infl_{x}(f_{n}, \mathscr{P}^{k}) = \P_{p, \frac{1}{k}}\left[ f_{n}(\eta_{0}, \mathscr{P}) \neq f_{n}(\eta_{0}^{x}, \mathscr{P})\middle|\mathscr{P}^{k} \right],
\end{equation}
and
\begin{equation}
\infl_{(x,s)}(f_{n}, \mathscr{P}^{k}) = \P_{p, \frac{1}{k}}\left[ f_{n}(\eta_{0}, \mathscr{P}) \neq f_{n}(\eta_{0}, \mathscr{P}^{(x,s)})\middle|\mathscr{P}^{k} \right],
\end{equation}
where $\eta_{0}^{x}$ and $\mathscr{P}^{(x,s)}$ are obtained from $\eta_{0}$ and $\mathscr{P}$ by exchanging the entries at $x$ and $(x,s)$, respectively. 

\par The crossing functions $f_{n}$ are monotone non-decreasing in the space variables $\eta_{0}$. Furthermore, the set $\bigcup_{y\in R_n}C_{k,t}^{\leftarrow}(y)$ comprised of vertices whose opinions at time~$0$ can influence the output of~$f_n(\eta_0,\mathscr{P})$ is almost surely finite.
Classical arguments then show that Russo's Formula applies to the derivative with respect to $p$ and one obtains
\begin{equation}
\label{eq:russo}
\frac{\partial}{\partial p}\E_{p,\frac{1}{k}}\left[f_n(\eta_{0},\mathscr{P})\middle| \mathscr{P}^{k}\right] = \sum_{x \in \Z^2} \infl_{x}(f_n, \mathscr{P}^{k}).
\end{equation}
Since
\begin{equation}
\left|\frac{\partial}{\partial p}\E_{p,\frac{1}{k}}\left[f_n(\eta_{0},\mathscr{P})\middle| \mathscr{P}^{k}\right]\right| \leq \bigg|  \bigcup_{y\in R_n}C_{k,t}^{\leftarrow}(y) \bigg|,
\end{equation}
as a direct consequence of the bounded convergence Theorem and Proposition~\ref{prop:col}, it is possible to conclude
\begin{equation}\label{eq:russo_inside_expectation}
\begin{split}
\frac{\partial}{\partial p}\E_{p,t}\left[f_n(\eta_{0},\mathscr{P})\right] 
&=
\E_{kt}\left[\frac{\partial}{\partial p}\E_{p,\frac{1}{k}}\left[f_n(\eta_{0},\mathscr{P})\middle| \mathscr{P}^{k}\right]\right] 
\\
&= \E_{kt}\bigg[\sum_{x \in \Z^2} \infl_{x}(f_n, \mathscr{P}^{k})\bigg].
\end{split}
\end{equation}

%%%%%
\newconstant{c:piv}
%%%%%

\par Regarding pivotality of clock ticks, we present a proposition that allows us to relate it to space-pivotality, provided we are in the event where the collection $\mathscr{P}^{k}$ is well behaved. Recall that $R_{n}=[1,n]^{2}$ and that $C_{k,t}^{\rightarrow}(x)$ denotes the future cone of light of the vertex $x$ associated to the collection of clocks~$\mathscr{P}^{k}$. For $\epsilon >0$, consider the event
\begin{equation}\label{eq:large_cone}
E(\epsilon) = \left\{\begin{array}{c} \text{there exists } x \in [-(n-1),2n]^{2} \text{ such that} \\ C_{k,t}^{\rightarrow}(x) \cap \partial B(x,\epsilon \log n) \neq \emptyset \end{array} \right\}.
\end{equation}
Our next proposition relates time-pivotality to space-pivotality, provided we are in the event $E(\epsilon)^{c}$.
\begin{prop}\label{prop:piv}
Given $k \geq 2$ and $p \in (0,1)$, there exists a positive constant $\useconstant{c:piv}=\useconstant{c:piv}(k,p)>0$ such that the following holds: for any $\mu>0$, there exists $\epsilon>0$ such that, for any bit associated to~$(x,s) \in \{x\} \times \mathscr{P}_{x}^{k}$,
\begin{equation}
\infl_{(x,s)}(f_n, \mathscr{P}^{k}) \mathbf{1}_{E(\epsilon)^{c}}(\mathscr{P}^{k}) \leq \useconstant{c:piv} n^{\mu} \sum_{y  \, \in \, \partial B(x, 3\epsilon \log n)} \infl_{y}(f_n, \mathscr{P}^{k}).
\end{equation}
Furthermore, if $p$ varies in a compact subset of $(0,1)$, the value of $\epsilon$ and $\useconstant{c:piv}$ can be chosen to be uniformly positive and bounded.
\end{prop}

\begin{proof}
Observe first that $|\partial B(x,3\epsilon \log n)| \leq 24 \epsilon \log n +8$.

Fix a configuration $\mathscr{P}^{k}$ in $E(\epsilon)^{c}$ and assume that the presence of the clock tick $(x,s)$ is pivotal. This can happen in two ways: first, it might be that adding the clock tick allows us to obtain a crossing, while, with the removal of such clock tick, no open crossings exist. The second possibility is the opposite: the addition of the clock tick prevents the existence of a crossing, while its removal implies on the presence of a crossing. We will consider only the first case, since the second can be treated similarly.

When the clock tick is present in the configuration (which we can assure by paying a finite multiplicative factor of $k$ in the probabilities), all possible crossings of the square $R_{n}$ intersect the cone of light $C_{k,t}^{\rightarrow}(x)$. In particular, since the clock tick is pivotal and we are in the event $E(\epsilon)^{c}$, these crossings necessarily intersect the box $B(x, 3\epsilon \log n)$. Hence, if we declare all vertices in $\partial B(x,3\epsilon \log n)$ as closed at time zero, no crossing can be found at time $t$. \emph{This is because ``monochromatic'' nearest-neighbor cycles are stable in the majority dynamics}. Every vertex in a ``monochromatic'' cycle is surrounded by at least two neighbors of the same opinion, and therefore its opinion remains forever unchanged. This defines a $2^{|\partial B(x,3\epsilon \log n)|}$-to-one map of the initial configurations.

We now proceed by successively changing each entry in~$(\eta_0(y))_{y\in\partial B(x,3\epsilon \log n)}$ which is one to zero. After all changes are performed, we obtain a configuration that has no crossing at time $t$. In particular, at some step, one of the entries of~$(\eta_0(y))_{y\in\partial B(x,3\epsilon \log n)}$ is space-pivotal for the configuration. Since in order to perform each of these changes we need to pay a multiplicative factor in the probabilities that is bounded from above by $\left(p \wedge (1-p) \right)^{-1}$, we can estimate
\begin{multline}
\infl_{(x,s)}(f_n, \mathscr{P}^{k}) \mathbf{1}_{E(\epsilon)^{c}}(\mathscr{P}^{k}) \\
\leq k 2^{|\partial B(x,3\epsilon \log n)|}\left(p \wedge (1-p) \right)^{-(24 \epsilon \log n +8)} \sum_{y  \, \in \, \partial B(x,3\epsilon \log n)} \infl_{y}(f_n, \mathscr{P}^{k}),
\end{multline}
where the factor $2^{|\partial B(x,3\epsilon \log n)|}$ comes from the cardinality of the pre-image of the mapping constructed. The proof is completed by choosing $\epsilon>0$ small enough.
\end{proof}

\section{Low-revealment algorithms}\label{sec:alg}
~
\par In order to apply the OSSS inequality to the crossing functions, we need to develop an algorithm that determines the existence of such crossings in the quenched case, when the realization of $\mathscr{P}^{k}$ is fixed, and bound its revealment. This is the goal of this section, where we define an algorithm with the desired properties and provide bounds on its revealment.

\par We begin by presenting the algorithm we will study. This algorithm will be a simple exploration process: we start with a random vertical line contained in the rectangle and query the opinion at time $t$ of all vertices that are in the given line. When we have this realization, we start exploring the components of open vertices that intersect this line. The existence of a crossing is equivalent to the existence of an open component that intersects this line and connects both sides of the rectangle.

\par For the rest of this subsection, we fix a realization $\mathscr{P}^{k}$ of the denser collection of clock ticks. Since we are working with a fixed realization of~$\mathscr{P}^{k}$, the sets $C_{k,t}^{\leftarrow}(x)$ are not random and depend only on the collection~$\mathscr{P}^{k}$. Of course, when we reveal the realization of $\mathscr{P}_{y}$ together with $\eta_{0}(y)$, for all $y \in C_{k,t}^{\leftarrow}(x)$, we can determine $\eta_{s}(x)$, for all $s \in [0,t]$. In view of this, whenever we \emph{query} the state of a vertex $x \in \Z^2$, we observe the initial opinions and selection of clock ticks for all vertices $y \in C_{k,t}^{\leftarrow}(x)$.

\par We are now in position to present the algorithm we will consider. Recall that $R_{n}=[1,n]^2$ and the notation $\Lambda = \Z^2 \cup \{(x,s): x \in \Z^2, \, s \in \mathscr{P}^{k}_{x} \cap [0,t]\}$.
\begin{algorithm}\caption{(Existence of a horizontal open crossing)}\label{alg:crossing}
\begin{algorithmic}[1]
\State \textbf{Input:} $\mathscr{P}^{k}$ and $(\eta_{0}, \mathscr{P}) \in \{0,1\}^{\Lambda}$.
\State If there exists $x \in R_n$ and $y \in C_{k,t}^{\leftarrow}(x)$ such that $\pnorm{1}{x-y} \geq \log n$, query all vertices of $R_n$.
\State Choose an integer point $\ell$ uniformly in the set $\left[1, n\right]\cap \Z$.
\State Query all vertices of $R_{n}$ whose first space-coordinate is $\ell$, and declare these vertices as explored.
\State Proceed to query all vertices that are neighbors to an open explored vertex, and declare all these vertices explored.
\State Repeat Step 5 until all open connected components inside $R_n$ that intersect $\{\ell\} \times \Z$ are discovered. If there exists a connected open component inside $R_n$ that connects $\{1\} \times \Z$ to $\{n\} \times \Z$, return 1. Otherwise, return~0.
\end{algorithmic}
\end{algorithm}

\par Notice that Algorithm~\ref{alg:crossing} clearly determines the existence of open crossings, since any open crossing intersects any vertical line $\big(\{\ell\}\times \Z \big) \cap R_{n}$. Furthermore, one can define an analogous algorithm that determines the existence of a closed vertical $*$-crossing of the box. When analyzing the revealment of the algorithm, we will consider Algorithm~\ref{alg:crossing} for $p \leq p_{c}(t)$ and its alternative formulation in terms of closed vertical $*$-crossings for $p > p_{c}(t)$.

\par We now proceed to bound the revealment of Algorithm~\ref{alg:crossing} (the bound on the alternative version is obtained analogously). Observe first that the revealment depends only on the sites $y \in \Z^2$, since we reveal all clock ticks of a given site $y$ at once, together with its initial opinion. We can therefore talk about the revealment of a site $y \in \Z^2$. Given a vertex $y \in \Z^{2}$, there are three different possibilities that might lead us to reveal it. The first case that comes from Step 2 in the algorithm is when, for some $x \in R_{n}$, $C^{\leftarrow}_{k,t}(x)$ is large. Second, it might be the case that $y \in C_{k,t}^{\leftarrow}(z)$, for some site $z$ in the vertical line segment $\big(\{\ell\} \times \Z \big)\cap R_{n}$. Finally, there is the case when $y \in C_{k,t}^{\leftarrow}(z)$ and some vertex adjacent to $z$ is connected to the selected vertical line segment by an open path.

\par In order to bound the revealment, we consider each of the three cases separately. The first and second cases can be easily controlled. As for the third case, we need finer estimates given by the one-arm estimates provided by Proposition~\ref{prop:one_arm}.

\begin{prop}\label{prop:revealment}
Let $\mathcal{A}$ denote the Algorithm~\ref{alg:crossing}, and let~$\mathcal{A}^*$ denote the analogous algorithm that looks for vertical closed $*$-crossings. Consider the revealments
\begin{equation}
\delta_\mathcal{A}(\mathscr{P}^k):=\sup_{x\in R_n} \delta(\mathcal{A},x);\quad 
\delta_{\mathcal{A}^*}(\mathscr{P}^k):=\sup_{x\in R_n} \delta({\mathcal{A}^*},x)
\end{equation}
There exist $\nu>0$ and $k_{0}>0$ such that, for all $k \geq k_{0}$, there exists $n_{0}=n_{0}(k)$ such that, if $n \geq n_{0}$ and $p \leq p_{c}(t)$, then
\begin{equation}
\label{eq:revealmentA}
\P_{kt}\left[\delta_\mathcal{A}(\mathscr{P}^k) > n^{-\nu} \right] \leq n^{-50},
\end{equation}
and if $p \geq p_{c}(t)$, then
\begin{equation}
\label{eq:revealmentA*}
\P_{kt}\left[\delta_{\mathcal{A}^{*}}(\mathscr{P}^k) > n^{-\nu} \right] \leq n^{-50}.
\end{equation}
\end{prop}

\begin{proof}
We will prove Equation~\eqref{eq:revealmentA}, \eqref{eq:revealmentA*} following the same reasoning. We examine separately the revealment of bits. First, we consider the case when $C_{k,t}^{\leftarrow}(x)$ is large, for some $x \in [1,n]^{2}$. Define the event
\begin{equation}
A = \left\{\begin{array}{c} \text{there exists $x \in R_{n}$ such that} \\ C_{k,t}^{\leftarrow}(x) \cap B(x,\log n) \neq \emptyset\end{array} \right\},
\end{equation}
and observe that Lemma~\ref{prop:col} implies
\begin{equation}
\P_{kt}\left[A\right] \leq n^{2}e^{-\frac{1}{8} \log n \log \log n}.
\end{equation}
Second, consider the event
\begin{equation}
B = \left\{\begin{array}{c} \text{there exists $x \in R_{n}$ such that} \\ \P_{p, \frac{1}{k}}\left[\arm_{\sqrt n}(x, \eta_{0}, \mathscr{P})\middle|\mathscr{P}^{k}\right] \geq n^{-\nu'} \end{array} \right\},
\end{equation}
where $\nu'$ is obtained from Proposition~\ref{prop:one_arm} by choosing $\gamma=100$, and observe that
\begin{equation}
\P_{kt}\left[B\right] \leq n^{2}n^{-100} = n^{-98}.
\end{equation}

We now bound the revealment on the event $A^{c} \cap B^{c}$. In this case, we split the revealment in two cases. Either the distance from the site $x$ to the random selected line is smaller then $2\sqrt{n}$, which is unlikely due to the randomness in selecting the line, or $x \in C_{k,t}^{\leftarrow}(y)$, for some $y$ such that a neighbor of it is connected to the random line by an open path. Since we are in the event $A^{c}$, we may assume that $y$ is close to $x$ and hence that, in the last case, $\arm_{\sqrt n}(x,\eta_{0},\mathscr{P})$ holds. This leads to the bound
\begin{equation}
\begin{split}
\lefteqn{\delta_\mathcal{A}(\mathscr{P}^k) \mathbf{1}_{A^{c} \cap B^{c}}(\mathscr{P}^{k}) }\phantom{******}
\\
&\leq \left(\max_{x \in R_{n}} \P_{p, \frac{1}{k}}\left[\arm_{\sqrt n}(x, \eta_{0}, \mathscr{P})\middle|\mathscr{P}^{k}\right]\right)\mathbf{1}_{A^{c} \cap B^{c}}(\mathscr{P}^{k})  + \frac{4\sqrt{n}}{\frac{n}{3}} \\
& \leq n^{-\nu'} + \frac{12}{n^{\sfrac{1}{2}}} \leq n^{-\nu},
\end{split}
\end{equation}
if $\nu$ is small enough and~$n$ large enough. In particular, we obtain from Proposition~\ref{prop:one_arm}, by choosing~$k$ and~$n$ sufficiently large,
\begin{equation}
\begin{split}
\P_{kt}\left[\delta_\mathcal{A}(\mathscr{P}^k) > n^{-\nu} \right]& \leq \P_{kt}\left[A \cup B\right] \\
& \leq n^{2}e^{-\frac{1}{8} \log n \log \log n}+n^{-98} \leq n^{-50},
\end{split}
\end{equation}
concluding the proof.
\end{proof}

\section{Sharp thresholds}\label{sec:thresholds}
~
\par In this section,we combine the results from the previous sections to conclude the proof of Theorem\ref{t:sharp_thresholds}. As already mentioned in Remark~\ref{remark:lambda=1}, we consider only the case $\lambda=1$.
\begin{proof}[Proof of Theorem~\ref{t:sharp_thresholds}.]
Given $\alpha>0$, consider the interval
\begin{equation}
I_{\alpha}(n)=\left\{p \in \left[\frac{1}{10},\frac{9}{10}\right]: \P_{p,t}[H(n,n)] \in [\alpha, 1-\alpha]\right\}.
\end{equation}
Our goal is to prove that the length of this interval is bounded by $cn^{-\gamma}$, for some positive constants $c=c(\alpha)$ and $\gamma$ that does not depend on $\alpha$. This is enough to conclude the proof, once we know that $p_{c}(t) \in I_{\alpha}(n)$, for all $n \in \N$, provided $\alpha$ is small enough.

We begin by introducing the event where the process is well behaved inside the box $R_{n}$. Recall the definition of the event $E(\epsilon)$ in~\eqref{eq:large_cone} and consider
\begin{equation}
A(\epsilon)=E(\epsilon) \cup \left\{|\mathscr{P}^{k}_{x}| \geq \log n, \text{ for some } x \in [-(n-1),2n]^{2} \right\}.
\end{equation}
Notice that, as a consequence of Proposition~\ref{prop:col} and standard bounds on the tail of the Poisson distribution, we obtain
\begin{equation}
\P_{kt}[A(\epsilon)] \leq 10 n^{2}\exp\left\{-\frac{\epsilon}{8} \log n \log \left(\epsilon \log n\right)\right\}
\end{equation}
if $n$ is large enough, depending on $k$ and $t$.

Given $p \in I_{\alpha}(n)$, consider the events
\begin{equation}
B(p)=\left\{ \P_{p,\frac{1}{k}}[f_{n}(\eta_{0}, \mathscr{P})=1|\mathscr{P}^{k}] \notin \left(\frac{\alpha}{2}, 1-\frac{\alpha}{2}\right)\right\}
\end{equation}
and
\begin{equation}
C = \left\{
\begin{array}{c}
\delta_\mathcal{A}(\mathscr{P}^k) \geq n^{-\nu} \text{ for some } p \in I_{\alpha}(n)\cap(0,p_c(t)];
\\
\delta_\mathcal{A^*}(\mathscr{P}^k) \geq n^{-\nu} \text{ for some } p \in I_{\alpha}(n)\cap(p_c(t),1).
\end{array}
\right\},
\end{equation}
where $\mathcal{A}$ denotes Algorithm~\ref{alg:crossing}, $\mathcal{A}^*$ denotes the analogous algorithm that looks for vertical closed $*$-crossings, and $\nu>0$ is given by Proposition~\ref{prop:revealment}. Here we observe that the revealment of our algorithm (or its analogue) is monotone in~$p$, since it is related to connection probabilities. This can be used to bound the probability of the above event, by considering only the case $p=p_{c}(t)$.
We claim that, for each $p \in I_{\alpha}(n)$,
\begin{equation}
\P_{kt}\left[B(p) \cup C\right] \leq \frac{4}{\alpha^{2}k}+2n^{-50}.
\end{equation}
The above bound follows partly from Proposition~\ref{prop:revealment} and partly from a reasoning analogous to the proof of Lemma~\ref{lemma:quenched_cir}. If we take $k$ large enough, and $n$ large depending on $k$ and $t$, we have
\begin{equation}
\P_{kt}[A(\epsilon) \cup B(p) \cup C] \leq \frac{1}{2}.
\end{equation}

We now use the OSSS inequality in the quenched setting. We assume that $p\leq p_c(t)$, the other case following analogously. If
\[
\mathscr{P}^{k} \in \left(A(\epsilon) \cup B(p) \cup C \right)^{c},
\]
we can use Proposition~\ref{prop:piv} with $\mu=\frac{\nu}{2}$ and Russo's Formula~\eqref{eq:russo} to estimate
\begin{equation}
\nonumber
\begin{split}
\lefteqn{\var\Big( f_{n}(\eta_{0}, \mathscr{P}) \Big|\mathscr{P}^{k}\Big)}
\phantom{**}
\\
&\leq \sum_{x} \delta_\mathcal{A}(\mathscr{P}^k)\Bigg(\infl_{x}\left(f_{n}, \mathscr{P}^{k}\right)+\sum_{s \in \mathscr{P}^{k}_{x}}\infl_{(x,s)}\left(f_{n}, \mathscr{P}^{k}\right)\Bigg) \\
&\leq \sum_{x} \delta_\mathcal{A}(\mathscr{P}^k)\Bigg(\infl_{x}\left(f_{n}, \mathscr{P}^{k}\right)+\useconstant{c:piv} n^{\frac{\nu}{2}}|\mathscr{P}^{k}_{x}|\sum_{y \in\partial B(x, 3\epsilon \log n)}\infl_{y}\left(f_{n}, \mathscr{P}^{k}\right)\Bigg) \\
&\leq n^{-\frac{\nu}{2}}\sum_{x} \infl_{x}\left(f_{n}, \mathscr{P}^{k}\right) \Big(1+\useconstant{c:piv} \log n \Big|\partial B(x, 3\epsilon \log n)\Big|\Big) \\
&\leq 25 \useconstant{c:piv} n^{-\frac{\nu}{2}}\log^{2} n\sum_{x} \infl_{x}\left(f_{n}, \mathscr{P}^{k}\right) \\
&\leq 25 \useconstant{c:piv} n^{-\frac{\nu}{2}}\log^{2} n \frac{\partial}{\partial p} \P_{p, \frac{1}{k}}\left[f_{n}(\eta_{0}, \mathscr{P})=1\middle|\mathscr{P}^{k}\right] \\
&\leq n^{-\frac{\nu}{3}}\frac{\partial}{\partial p} \P_{p, \frac{1}{k}}\left[f_{n}(\eta_{0}, \mathscr{P})=1\middle|\mathscr{P}^{k}\right],
\end{split}
\end{equation}
if $n$ is large enough.

In particular, for $p \in I_{\alpha}(n)$, using the fact that~$f_{n}(\eta_{0}, \mathscr{P})$ is a~$\ber$ variable which is increasing in the intensity of~$\eta_0$ and~\eqref{eq:russo_inside_expectation},
\begin{equation}
\begin{split}
\frac{\partial}{\partial p} \P_{p, t}\left[H(n,n)\right] & = \frac{\partial}{\partial p} \E_{kt} \left[ \P_{p, \frac{1}{k}}\left[f_{n}(\eta_{0}, \mathscr{P})=1\middle|\mathscr{P}^{k}\right]\right] \\
& \geq \E_{kt}\left[ \frac{\partial}{\partial p} \P_{p, \frac{1}{k}}\left[f_{n}(\eta_{0}, \mathscr{P})=1\middle|\mathscr{P}^{k}\right] \mathbf{1}_{\left(A(\epsilon) \cup B(p) \cup C \right)^{c}}\left(\mathscr{P}^{k}\right) \right] \\
& \geq n^{\frac{\nu}{3}}\E_{kt}\left[ \var \Big( f_{n}(\eta_{0}, \mathscr{P}) \Big|\mathscr{P}^{k}\Big) \mathbf{1}_{\left(A(\epsilon) \cup B(p) \cup C \right)^{c}}\left(\mathscr{P}^{k}\right) \right] \\
& \geq n^{\frac{\nu}{3}}\frac{\alpha^{2}}{4} \P_{p,t}\left[\left(A(\epsilon) \cup B(p) \cup C \right)^{c}\right] \geq n^{\frac{\nu}{3}}\frac{\alpha^{2}}{8}.
\end{split}
\end{equation}

This implies
\begin{equation}
1 \geq \int_{I_{\alpha}(n)}\frac{\partial}{\partial p} \P_{p, t}\left[H(n,n)\right] \, \d p \geq n^{\frac{\nu}{3}}\frac{\alpha^{2}}{8}|I_{\alpha}(n)|,
\end{equation}
which gives the bound
\begin{equation}
|I_{\alpha}(n)| \leq \frac{8}{\alpha^{2}}n^{-\frac{\nu}{3}},
\end{equation}
and concludes the proof.
\end{proof}

\section{One-arm estimates}\label{sec:one_arm}
~
\par The goal of this section is to conclude the proof of the quenched one-arm estimates stated as Proposition~\ref{prop:one_arm}.
\begin{proof}
We will work on the event where all cones of light are well behaved. For each $n$, define
\begin{equation}
E_n:=\left\{C_{k,t}^{\leftarrow}(x) \cap \partial B\big(x, n^{\sfrac{1}{4}}\big) = \emptyset\text{ for every }x \in B\big(0, n^{\sfrac{1}{2}}\big) \right\}.
\end{equation}
Proposition~\ref{prop:col} implies, for sufficiently large~$n$,
\begin{equation}\label{eq:one_arm_1}
\P_{kt}[E_{n}^{c}] \leq 16n e^{-\frac{1}{32}n^{\sfrac{1}{4}} \log n}.
\end{equation}

Consider the collection of indices
\begin{equation}
J=\left\{j \in 2\N: n^{\frac{1}{4}} \leq 3^{j}n^{\frac{1}{4}} \leq n^{\frac{1}{2}}\right\},
\end{equation}
and, for $j \in J$, denote by $A_{j}$ the set of vertices
\begin{equation}
A_{j}=B\left(0, 2 \cdot 3^{j+1} n^{\sfrac{1}{4}}\right) \setminus B\left(0, 2\cdot 3^{j-1} n^{\sfrac{1}{4}}\right)
\end{equation}
and recall the definition of $\cir^{*}(m)$ immediately after~\eqref{eq:cir}. Notice that, on $E_{n}$, $\cir^{*}\left(3^{j}n^{\sfrac{1}{4}}\right)$ depends on $(\eta_{0}(x), \mathscr{P}_{x})$ only for $x \in A_{j}$. In particular, on $E_{n}$, we can use independence to estimate
\begin{equation}\label{eq:one_arm_2}
\begin{split}
\lefteqn{\P_{p, \frac{1}{k}}\left[\arm_{\sqrt n}(\eta_{0}, \mathscr{P})|\mathscr{P}^{k} \right]\mathbf{1}_{E_{n}}(\mathscr{P}^{k}) }\phantom{******}
\\
&\leq \P_{p, \frac{1}{k}}\left[\bigcap_{j \in J}\cir^{*}\left(3^{j}n^{\sfrac{1}{4}}\right)^{c}\middle|\mathscr{P}^{k} \right]\mathbf{1}_{E_{n}}(\mathscr{P}^{k}) \\
& = \prod_{j \in J}\P_{p, \frac{1}{k}}\left[\cir^{*}\left(3^{j}n^{\sfrac{1}{4}}\right)^{c}\middle|\mathscr{P}^{k} \right]\mathbf{1}_{E_{n}}(\mathscr{P}^{k}).
\end{split}
\end{equation}
Consider now the event
\begin{equation}
D_{j}=\left\{\P_{p, \frac{1}{k}}\left[\cir^{*}\left(3^{j}n^{\sfrac{1}{4}}\right)\middle|\mathscr{P}^{k} \right] \geq \frac{\useconstant{c:cir}}{2} \right\},
\end{equation}
and denote by $D$ the event where $D_{j}$ holds for at least half of the indices $j \in J$.
From~\eqref{eq:one_arm_2}, we obtain
\begin{equation}
\begin{split}
\P_{p, \frac{1}{k}}\left[\arm_{\sqrt n}(\eta_{0}, \mathscr{P})\middle|\mathscr{P}^{k} \right] & \mathbf{1}_{E_{n}\cap D}(\mathscr{P}^{k}) \\
& \leq \prod_{j \in J}\P_{p, \frac{1}{k}}\left[\cir^{*}\left(3^{j}n^{\sfrac{1}{4}}\right)^{c}\middle|\mathscr{P}^{k} \right]\mathbf{1}_{E_{n} \cap D}(\mathscr{P}^{k}) \\
& \leq \left(1-\frac{\useconstant{c:cir}}{2}\right)^{\frac{|J|}{2}} \leq n^{-\nu},
\end{split}
\end{equation}
for some $\nu$ small enough, since $|J|$ is of order $\log n$. This implies that
\begin{equation}\label{eq:one_arm_2.5}
\P_{kt}\Big[\P_{p, \frac{1}{k}}\left[\arm_{\sqrt n}(\eta_{0}, \mathscr{P})|\mathscr{P}^{k}\right] \geq n^{-\nu} \Big] \leq \P_{kt}\left[E_{n}^{c} \cup D^{c}\right],
\end{equation}
so it remains to bound the right hand side probability above.

We begin by estimating
\begin{equation}\label{eq:one_arm_3}
\P_{kt}\left[E_{n}^{c} \cup D^{c}\right] \leq \P_{kt}\left[E_{n}^{c}\right]+ 2^{|J|-1}\sup_{I} \P_{kt}\left[E_{n} \cap \bigcap_{j \in I} D_{j}^{c}\right],
\end{equation}
where the supremum is taken over all subsets of $J$ with at least $\frac{|J|}{2}$ indices. From Lemma~\ref{lemma:quenched_cir}, we obtain
\begin{equation}
\P_{kt}[D_{j}^{c}] \leq \frac{4}{\useconstant{c:cir}^{2} k},
\end{equation}
provided $k$ is taken large enough. For each $j \in J$, let $E_{n}(j)$ be an event analogous to $E_{n}$, but only observing the cone of light of vertices inside $B\left(0,3^{j+1}n^{\sfrac{1}{4}}\right) \setminus B\left(0,3^{j}n^{\sfrac{1}{4}}\right)$. The events $\left(E_{n}(j)\cap D _{j}^{c}\right)_{j \in J}$ are then independent, since they depend on disjoint parts of the graphical construction. From this, we obtain
\begin{equation}\label{eq:one_arm_4}
\begin{split}
\P_{kt}\left[E_{n} \cap \bigcap_{j \in I} D_{j}^{c}\right] & \leq \P_{kt}\left[\bigcap_{j \in I}E_{n}(j) \cap D_{j}^{c}\right] \\
& \leq \prod_{j \in I} \frac{4}{\useconstant{c:cir}^{2} k} \leq \left(\frac{4}{\useconstant{c:cir}^{2} k}\right)^{\frac{|J|}{2}},
\end{split}
\end{equation}
whenever $I \geq \frac{|J|}{2}$ and $k$ is large enough.

Combining Equations~\eqref{eq:one_arm_1},~\eqref{eq:one_arm_2.5},~\eqref{eq:one_arm_3}, and~\eqref{eq:one_arm_4} yields
\begin{equation}
\P_{kt}\Big[\P_{p, \frac{1}{k}}\left[\arm_{\sqrt n}(\eta_{0}, \mathscr{P})\middle|\mathscr{P}^{k}\right] \geq n^{-\nu} \Big] \leq n^{-\gamma},
\end{equation}
for all $n$ large enough, by further increasing the value of $k$ if necessary. This concludes the proof of the result.
\end{proof}

\section{Stretched-exponential decay of the one-arm event probability}\label{sec:decay}
~
%%%%%%
\newconstant{c:1armexpgen1}
\newconstant{c:1armexpgen2}
%%%%%%
In this section we will prove Theorem~\ref{t:exp_decay} using the results so far obtained. We will in fact prove a more general result, based on the proof of Theorem~$3.1$ of~\cite{pt}, which, together with a decoupling inequality and Theorem~\ref{t:sharp_thresholds}, will imply the desired rate of decay.

We first develop some notation needed before we state the result. Given $L\in\R_+$ and~$x\in\Z^d$, we define the subsets
\begin{equation}
\begin{split}
\label{eq:annulusdef}
C_x(L):=[0,L)^d + x,\quad\quad  D_x(L):=[-L,2L)^d\cap\Z^d + x.
\end{split}
\end{equation}
In accordance with~\eqref{eq:open_connection}, we denote by $\{A \longleftrightarrow B\}$ for the event where there exists a nearest-neighbor open path starting at~$A$ and ending at~$B$. For~$x\in\Z^d$, $L\in\R_+$, we define the annulus-crossing event
	\begin{equation*}
	A_x(L) := \{C_x(L){\longleftrightarrow} \Z^d\setminus D_x(L)\}.
	\end{equation*}
\begin{prop}
	\label{p:1armdecaygen}
	Let~$\tilde \P$ denote a probability distribution over~$\{0,1\}^{\Z^d}$, invariant under translations of~$\Z^d$. Assume that
	\begin{equation}
	\label{eq:1armexpgen2}
	\liminf_{L\to \infty} \tilde \P \left[A_0(L) \right]<\frac{1}{d^{2}\cdot 7^d},
	\end{equation}
	and that there exists a positive constant~$\useconstant{c:1armexpgen1}>0$ such that, for every~$L,R\in\R_+$ and every~$x,y\in \Z^d$ with~$\|x-y\|_\infty \geq 3L+R$, we have
	\begin{equation}\label{eq:1armexpgen3}
	\left| \tilde \P\left[ A_x(L) \cap A_y(L)      \right]  -  \tilde\P\left[ A_x(L)       \right]  \tilde\P\left[ A_y(L)      \right]     \right| \leq L^{2d} \exp\left\{ -f(R)      \right\},
	\end{equation}
	where~$f:\R_+\to\R_+$ is a non-decreasing function such that
	\begin{equation}
	\begin{split}
	\label{eq:1armexpgenf}
	\liminf_{R\to \infty} \frac{f(R)}{R\log R}\geq \useconstant{c:1armexpgen1}.
	\end{split}
	\end{equation}
	Then, for every~$\varepsilon>0$ there exists a positive constant $\useconstant{c:1armexpgen2}=\useconstant{c:1armexpgen2}(\varepsilon)>0$ such that, for~$n\in\N$,
	\begin{equation}\label{eq:1armexpgen4}
	\tilde \P\left[ \{0\} \longleftrightarrow \partial B(0,n) \right]  \leq \useconstant{c:1armexpgen2}^{-1}\exp\left\{
	-\useconstant{c:1armexpgen2}\frac{n}{(\log n)^\varepsilon}
	\right\}.
	\end{equation}
\end{prop}

\begin{proof}
	
	The proof is based on the proof of Theorem~$3.1$ of~\cite{pt}, specifically, the proof of Equation~$(3.5)$. Since in our case no sprinkling argument is needed in order to obtain a decoupling inequality, the argument here will be simpler.
	
	The proof consists in a multiscale renormalisation argument. We start by inductively defining the sequence of scales~$(L_k)_{k\in\N}$. Given~$L_1\in\R_+$, which will chosen to be large, we let, for~$k\in\N$, 
	\begin{equation}
	\begin{split}
	\label{eq:Lkdef}
	L_{k+1}=2\left(1+\frac{1}{(k+5)^{1+\varepsilon}}\right)L_{k}.
	\end{split}
	\end{equation}
	One can easily check that
	\begin{equation}
	\begin{split}
	\label{eq:Lkbound3}
	L_1 2^{k-1} \leq L_k \leq C_\varepsilon L_1 2^{k-1} . 
	\end{split}
	\end{equation}
	for some constant~$C_\varepsilon>0$ depending on the exponent~$1+\varepsilon$.
	
	The proof of the theorem is based on an induction argument that bounds the probability
	\begin{equation}
	\begin{split}
	\label{eq:pkdef}
	p_k:=\tilde \P\left[A_0(L_k)\right].
	\end{split}
	\end{equation}
	Recalling the sets defined in~\eqref{eq:annulusdef}, note that, for~$k\geq 1$, there exist two collection of points~$\{x_i^k\}_{i=1}^{3d}$ and~$\{y_j^k\}_{j=1}^{2d\cdot 7^{d-1}}$ such that
	\begin{equation}
	\begin{split}
	\label{eq:xiyiproperty}
	&C_0(L_{k+1})=\cup_{i=1}^{3d}C_{x_i}(L_{k}),
	\\
	&\left(\cup_{j=1}^{2d\cdot 7^{d-1}}C_{y_j}(L_{k})\right)\cap D_0(L_{k+1})=\emptyset,
	\\
	&\partial(\Z^d\setminus D_0(L_{k+1}))\subset\cup_{j=1}^{2d\cdot 7^{d-1}}C_{y_j}(L_{k}).
	\end{split}
	\end{equation}
	Properties~\eqref{eq:xiyiproperty} then imply (see Figure~\ref{f:multiscale})
	\begin{equation}
	\begin{split}
	\label{eq:akinduc}
	A_0(L_{k+1})\subset \bigcup_{\substack{i \leq 3 d \\ j \leq 2d\cdot 7^{d-1}}} A_{x_i^k}(L_{k})\cap A_{y_j^k}(L_{k}).
	\end{split}
	\end{equation}
	
	\begin{figure}[ht]
		\centering
		\includegraphics[scale =.7]{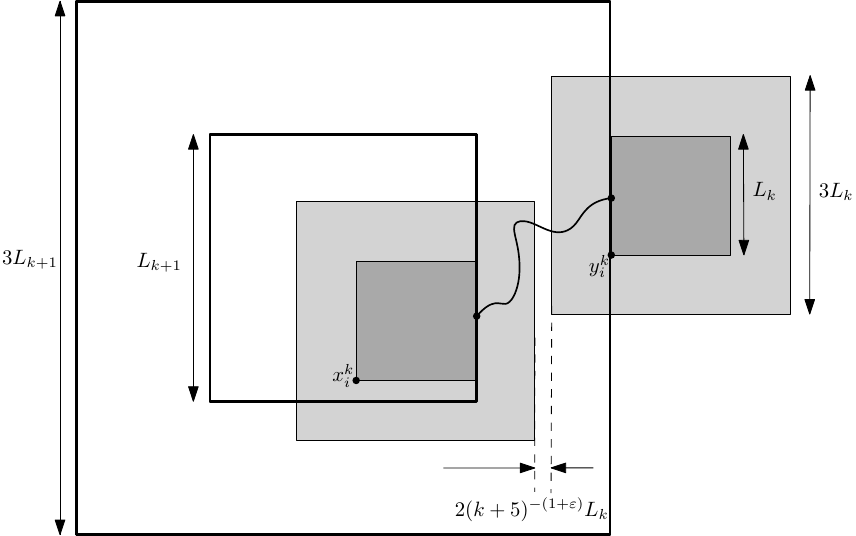}
		\vspace{0.5cm}
		\caption{The ``cascading'' nature of the events~$A_x(L_k)$. }
		\label{f:multiscale}
	\end{figure}
	
It is also elementary to check that the distance between $D_{x_i^k}(L_k)$ and $D_{y_j^k}(L_k)$ is greater than~$2(k+5)^{-(1+\varepsilon)}L_k$ uniformly in~$i$ and~$j$. Property~\eqref{eq:1armexpgen3} then implies, together with the equation above and the translation invariance of~$\tilde \P$, for~$k\geq 1$,
	\begin{equation}
	\begin{split}
	\label{eq:pkinduc}
	p_{k+1}\leq d^{2} \cdot 7^d \left( 
	p_k^2 + L_k^{2d}\exp\left\{  -   f\left(     \frac{2L_k}{(k+5)^{1+\varepsilon}}      \right)       \right\}
	\right).
	\end{split}
	\end{equation}
	Proceeding in the same way as in~\cite{pt}, one proves by induction that Equations~\eqref{eq:1armexpgen2} and~\eqref{eq:1armexpgen3} imply, for suitably chosen real numbers~$h_1,h_2>0$,
	\begin{equation}
	\begin{split}
	\label{eq:pkinduc3}
	p_{k}\leq 
	\exp\left\{ -h_1 -h_2\frac{2^k}{k^{\varepsilon}} \right\}.
	\end{split}
	\end{equation}
	Note then that, for~$n\in[2L_k,2L_{k+1}]$, we have
	\begin{equation}
	\label{eq:Lkton}
	\left\{\{0\} \longleftrightarrow \partial B(0,n)\right\} \subseteq A_0(L_k),
	\end{equation}
	and therefore
	\begin{equation}
	\label{eq:Lkton2}
	\tilde \P[\{0\} \longleftrightarrow \partial B(0,n) ] \leq  \tilde \P[A_0(L_k)] \leq \exp\left\{ -h_1 -h_2 \frac{2^k}{k^\varepsilon} \right\}.
	\end{equation}
	Equation~\eqref{eq:Lkbound3} then implies the result, for a suitably chosen constant~$\useconstant{c:1armexpgen2}$.
\end{proof}

We can then finish the proof of the main result of this section.

\begin{proof}[Proof of Theorem~\ref{t:exp_decay}]
For~$p<p_c(t)$, Theorem~\ref{t:sharp_thresholds} implies
\begin{equation}
\limsup_{n}\P_{p,t}\left[  A_0(n)\right] \leq 4 \limsup_{n} \P_{p,t}[H(n,3n)] = 0.
\end{equation}
Proposition~\ref{prop:decoup} and basic properties of majority dynamics imply that~$\P_{p,t}$ satisfies the hypotheses of Proposition~\ref{p:1armdecaygen}, which then implies equation~\eqref{eq:exp_decay1}. Equation \eqref{eq:exp_decay2} follows from a completely analogous reasoning. 

Notice that in the event where the finite open cluster of the origin has diameter~$n$, one can find a closed vertex inside~$B(0,n)$ through which a closed $*$-path of length larger than~$n$ passes. Equation \eqref{eq:exp_decay2} and a union bound argument then yield Equation~\eqref{eq:exp_decay3}, and the analogous result for closed $*$-clusters containing the origin follows from~\eqref{eq:exp_decay1} and the same reasoning.
\end{proof}

\section{Further models}\label{sec:further_models}
~
\par As we already mentioned, our technique can be applied to other particle systems as long as some basic properties can be verified. In particular, we require equivalent formulations of Lemmas~\ref{lemma:variance_decay} and of Propositions~\ref{prop:one_arm},~\ref{prop:col} and~\ref{prop:piv}. Here, we extend our results for the voter model in the two-dimensional lattice $\Z^{2}$.

\par The voter model is very similar to majority dynamics, in the sense that it differs just in the way each vertex selects its new opinion once its clock ticks. In this case, the new opinion is selected randomly among the neighbors' opinions.

\par Once again, for each fixed time $t$, there exists a non-trivial critical parameter $p^{VM}_{c}(t) \in (0,1)$ for the existence of percolation at time $t$. Since, by applying Proposition~\ref{prop:col}, we can derive decoupling estimates that are uniform in the value of $p \in [0,1]$, these can be used to deduce taht $p_{c}^{VM}(t)$ is non-trivial , for each positive $t$, by standard renormalisation arguments.

\par The usual graphical construction of the voter model (see Remark~\ref{remark:voter_model}) can be modified exactly as we did in Section~\ref{sec:construction}, and Lemma~\ref{lemma:variance_decay} and Proposition~\ref{prop:one_arm} can be obtained from general results, as in the case of majority dynamics. Furthermore, we can apply the same proof to obtain Proposition~\ref{prop:col}. The most delicate part is in establishing a relation between time-pivotality and  space-pivotality.

\par Let us now describe how one approaches Proposition~\ref{prop:piv} here. In this case, we use the fact that the opinion of each vertex at any time $s \geq 0$ is a copy of one of the initial opinions that are contained in the past cone of light. Not only that, but changing this opinion at time zero implies that the opinion changes at time $s$. This last observation allows us to conclude that time-pivotality implies space-pivotality for some vertex in the cone of light. From this, we derive the bound
\begin{equation}
\infl_{(x,s)}(f_n, \mathscr{P}^{k}) \leq c \sum_{y  \, \in \, C_{k,t}^{\leftarrow}(x)} \infl_{y}(f_n, \mathscr{P}^{k}),
\end{equation}
for some positive constant $c>0$. This yields a version of Proposition~\ref{prop:piv} that can be used to conclude Theorem~\ref{t:sharp_thresholds} for the voter model.

\bibliographystyle{plain}
\bibliography{mybib}

\end{document}